\documentclass[10pt]{amsart}
\usepackage{amssymb,amsmath,amsrefs,mathtools}

\DeclareMathOperator\Arg{Arg}
\usepackage{xcolor}
\usepackage{physics}
\usepackage{verbatim}
\usepackage[hidelinks]{hyperref}
\newtheorem{thm}{Theorem}[section]
\newtheorem{lem}{Lemma}[section]
\newtheorem{deff}{Definition}[section]
\newtheorem{prop}{Proposition}[section]
\newtheorem{cor}{Corollary}[section]

\newtheorem{conj}{Conjecture}[section]
\theoremstyle{definition}

\newtheorem{rem}{Remark}

\def\Q {\mathbb{Q}}
\def\R {\mathbb{R}}

\def\o {\mbox{Orb}}
\def\uL {\mathbb{L}_{\text{ultra}}}

\def\b {\beta}
\def\o {\omega}

\def\a {\alpha}

\newcommand{\Z}{\mathbb{Z}}

\newcommand{\QQ}{\overline{\Q}}

\newcommand{\C}{\mathbb{C}}

\newcommand{\e}{\epsilon}
\newcommand{\pq}{p_k/q_k}

\newcommand{\g}{\gamma}

\newcommand{\cL}{\mathcal{L}}

\begin{document}

\title[The Transcendence of Power Towers of Liouville Numbers]{On the Transcendence of Power Towers of Liouville Numbers}

\author[D. Marques]{Diego Marques}
\address{Departamento de Matem\'atica\\
Universidade de Bras\'ilia\\
Bras\'ilia, DF\\
Brazil}
\email{diego@mat.unb.br}

\author[M. Oliveira]{Marcelo Oliveira}
\address{Departamento de Matem\'atica\\
Universidade Estadual do Maranhão\\
Brazil}
\email{marcelo\_mat@yahoo.com}

 \author[P. Trojovsk\' y]{Pavel Trojovsk\' y}
	\address{Faculty of Science\\
 University of Hradec Kr\'alov\'e\\
 Czech Republic}
	\email{pavel.trojovsky@uhk.cz}


\begin{abstract}
In this paper, among other things, we explicit a $G_{\delta}$-dense set of Liouville numbers, for which the triple power tower of any of its elements is a transcendental number.
\end{abstract}

\maketitle

\section{Introduction}\label{intro_section}

The beginning of the transcendental number theory happened in 1844 when Liouville \cite{liouville} showed that real algebraic numbers are not ``well-approximated"\ by rational numbers. By using this fact, Liouville explicitly, for the first time, provided examples of transcendental numbers (known as Liouville numbers now). We recall that a real number $\xi$ is called a {\it Liouville number} if there exist infinitely many rational numbers $p_k/q_k$, with $q_k>1$, such that
\begin{center}\label{cl}
$\displaystyle \left| \xi - \frac{p_k}{q_k} \right| < q_k^{-\o_k}$, $\forall k$,
\end{center}
for some sequence of positive real numbers $(\o_k)_k$ which tends to $\infty$ as $k\to \infty$. The first example of a Liouville number (and consequently, of a transcendental number) is the known {\it Liouville constant} defined by $\ell:=\sum_{n\geq 1}10^{-n!}$ (i.e., the decimal with $1$'s in each factorial position and $0$'s otherwise). The set of Liouville numbers is denoted by $\mathbb{L}$.

In his pioneering book, Maillet \cite[Chapitre III]{mai} discussed some arithmetic properties of Liouville numbers. One of them is that given a non-constant rational function $f$, with rational coefficients, if $\xi$ is a Liouville number, then so is $f(\xi)$. Thus, a natural problem concerns the possible arithmetic behavior of $f(\xi)$ when $\xi\in \mathbb{L}$ and $f$ is a transcendental function. In particular, it remains unsolved the question raised by Mahler \cite{mah3} about the existence of a transcendental entire function taking $\mathbb{L}$ into itself. However, there are some results on this subject. For instance, Chalebgwa and Morris \cite{bull} proved, that  $e^{\xi}$ is transcendental whenever $\xi\in \mathbb{L}$. Also, Marques and Oliveira \cite{MO} showed that $\a^{\xi}$ is transcendental, for all Liouville number $\xi$ and all algebraic number $\a\not\in \{0,1\}$. By using Baire's theorem (see \cite{Tang}), one has that for any non-constant function $f$, there is a $G_{\delta}$-dense set $\cL$ of Liouville numbers such that $f(\cL)\subseteq \mathbb{L}$. 

Set, inductively, $\exp[n](x) = \exp(\exp[n-1](x))$ and $\exp[0](x) = x$. In 2015, Marques and Moreira \cite{digu} defined an {\it ultra-Liouville number} $\xi$ as a real number for which, for every positive integer $k$, there exist infinitely many rational numbers $p/q$, with $q > 1$, such that
\[
0 < \left|\alpha - \frac{p}{q}\right| < \frac{1}{\exp[k](q)}.
\]

They proved the existence of a transcendental entire function $f$ for which $f(\mathbb{L}_{\text{ultra}})\subseteq \mathbb{L}_{\text{ultra}}$ (where $\mathbb{L}_{\text{ultra}}$ denotes the set of ultra-Liouville numbers). 

Very recently, Marques and Oliveira \cite{MO} showed that $\xi^{\xi}$ is a transcendental number whenever $\xi$ is an ultra-Liouville number (indeed, they provided a larger class of Liouville numbers with this property). 

We now define sub-classes of Liouville numbers $\xi$ (which satisfy \eqref{cl}) depending on the growth of the map $\nu\mapsto \lim_{k\to \infty}\o_k/q_k^{\nu}$. More specifically, we have the following definition:

\begin{deff}
A real number $\xi$ is called a {\it $\nu$-Liouville number} if $\nu$ is the infimum (possibly infinite) of the non-negative real numbers $\nu^*$ for which there exists an infinite sequence of rational numbers $(p_k/q_k)_k$ (with $q_k>1$) and a sequence of positive real numbers $(\o_k)_k$ such that
\begin{center}
$\displaystyle \left| \xi - \frac{p_k}{q_k} \right| < q_k^{-\o_k}$, $\forall k$,
\end{center}
and $\o_k/q_k^{\nu^*}$ tends to $\infty$ as $k\to \infty$. The set of all $\nu$-Liouville numbers is denoted by $\mathbb{L}^{(\nu)}$. 
\end{deff}

We note that $0$-Liouville numbers are precisely the Liouville numbers and $\mathbb{L}^{(\nu)}\subseteq \mathbb{L}$, for any $\nu \in [0,\infty]$. Also, $\mathbb{L}^{(\infty)}$ is a $G_{\delta}$-dense set since it contains the set of ultra-Liouville numbers (indeed, one has that $|\xi-p_k/q_k|<\exp(-\exp(q_k))=q_k^{-\omega_k}$, where $\omega_k:=\exp(q_k)/\log q_k$, satisfies $\lim_{k\to \infty}\omega_k/q_k^{\nu}=\infty$, $\forall \nu\geq 0$). For topological reasons, the set $\mathbb{L}^{(\nu)}$ is not $G_{\delta}$-dense whenever $\nu \in (0,\infty)$. So a natural question arises: are the sets $\mathbb{L}^{(\nu)}$ non-empty, for all $\nu\in (0,\infty)$? Our first result answers this question positively. More precisely,

\begin{thm}\label{main0}
The set $\mathbb{L}^{(\nu)}$ is uncountable, for any $\nu > 0$.
\end{thm}

The main goal of this paper is to provide some results on the arithmetic behavior of a triple power tower of numbers related to the previously defined sets. In order to state some conditions succinctly, we introduce the following terminology: a triple of polynomials $(P(X),Q(X),R(X))$ is said to be \textit{$\mathbb{K}$-admissible} (where $\mathbb{K}\subseteq \mathbb{C}$), when
\begin{itemize}
    \item[\mbox{(i)}] $P(X), Q(X)\in \mathbb{K}[X]$ and $R(X)\in \mathbb{Z}[X]$;
    \item[\mbox{(ii)}] $P(X)\not\equiv 0$, $Q(X)\not\equiv 0$ and $P(X)\not\equiv 1$;
    \item[\mbox{(iii)}] $P'(X)\equiv 0$ implies $Q'(X)\not\equiv 0$ and $R(X)\not\equiv 0$;
    \item[\mbox{(iv)}] $Q'(X)\equiv 0$ implies $P'(X)\not\equiv 0$;
    \item[\mbox{(v)}] $P'(X)\equiv 0,\ Q'(X)\equiv 0$ implies $R'(X)\not\equiv 0$.
\end{itemize} 
where, as usual, $T(X)\equiv c$ (resp., $T(X)\not\equiv c$) means that $T(X)=0$, for all $X$ (resp., $T(X)\neq c$, for some $X$).

Thus, our main result can be stated as follows:

\begin{thm}\label{main1}
Let $(P(X), Q(X), R(X))$ be a $\mathbb{Z}$-admissible triple of polynomials. If $\xi$ is a $\nu$-Liouville number such that $\nu > 6\deg R$, then 
\[
P(\xi)^{Q(\xi)^{R(\xi)}}
\]
is a transcendental number.
\end{thm}

\begin{rem}
    As discussed before, by supposing that the function $f:I\to \R$, defined by $f(X):=P(X)^{Q(X)^{R(X)}}$ is well-defined and non-constant in some open interval $I\subseteq \R$, then $f(\xi)$ is transcendental, for any $\xi$ belonging to an uncountable set of Liouville numbers. However, there is no information on the approximation structure of this set. Thus, our Theorem \ref{main1} provides such a structure that depends only on the degree of $R(X)$.
\end{rem}

As immediate consequences, we have

\begin{cor}\label{main2}
If $\xi$ is an $\infty$-Liouville number, then $P(\xi)^{Q(\xi)^{R(\xi)}}$ is a transcendental number, for any non-constant polynomials $P(X), Q(X), R(X)\in \Z[X]$.
\end{cor}

\begin{cor}\label{main3}
If $\xi\in \uL$, then $\xi^{\xi^{\xi}}$ is a transcendental number.
\end{cor}

Furthermore, by exploring the admissibility conditions, we obtain

\begin{cor}\label{main4}
Let $\nu>6$ be a real number. Then, for any $\nu$-Liouville number $\xi$ and for any integer $n\geq 2$, all the numbers
\[
n^{n^{\xi}},\ n^{\xi^{\xi}},\ n^{\xi^{n}},\ \xi^{n^{\xi}}, \xi^{\xi^{n}}
\]
are transcendental.
\end{cor}

The proofs presented here combine some algebraic and analytic tools together with a bit of Diophantine approximation and lower bounds for linear forms in logarithms.

\section{Proof of Theorem \ref{main0}}

For a given $\nu>0$, let $(s_k)_k$ be the sequence defined by $s_1=4$ and $s_{k+1}=\lfloor 10^{\nu s_k}\rfloor s_k+1$, for $k\geq 1$. We claim that the number $\xi:=\sum_{n=1}^{\infty}10^{-s_n}$ is a $\nu$-Liouville number.

In order to prove that, we start by defining coprime integers $A_k:=\sum_{n=1}^k10^{s_k-s_n}$ and $B_k:=10^{s_k}$. Since $A_k/B_k$ is a truncation of the series which defines $\xi$, we get
\begin{equation}\label{01}
    \abs{\xi-\frac{A_k}{B_k}}=\sum_{n\geq k+1}10^{-s_n}<\frac{2}{10^{s_{k+1}}}=\frac{2}{B_k^{s_{k+1}/s_k}}< \frac{1}{B_k^{\lfloor 10^{\nu s_k}s_k\rfloor}}.
\end{equation}

Note that $\lfloor 10^{\nu s_k}\rfloor/B_k^{\nu}=\lfloor B_k^{\nu}s_k\rfloor/B_k^{\nu}$ tends to infinite as $k\to \infty$ yielding that $\xi$ belongs to $\mathbb{L}^{(\tau)}$, for some $\tau\geq \nu$. To finish the proof, it suffices to show that $\xi$ is not a $(\nu+\epsilon)$-Liouville number, for all $\epsilon>0$. Indeed, by (\ref{01}), one has
\[
\abs{\xi-\frac{A_k}{B_k}}<\frac{1}{B_k^{\lfloor 10^{\nu s_k}s_k\rfloor}}<\frac{1}{2B_k^2},
\]
where we used that $\lfloor 10^{\nu s_k}s_k\rfloor\geq 10^{\nu s_k}s_k-1\geq s_k-1\geq 3$. Hence, a criterion due to Legendre yields that $A_k/B_k$ is a convergent, say $p_{r(k)}/q_{r(k)}$, of the continued fraction of $\xi=[b_0;b_1,b_2,\ldots]$. Note that $q_{r(k)}=B_k=10^{s_k}$. Also, for $r(k-1)<m\leq r(k)$, it holds that
\[
b_m\leq \displaystyle\prod_{i=r(k-1)+1}^{r(k)}b_i< \dfrac{q_{r(k)}}{q_{r(k-1)}}=q_{r(k-1)}^{-1+s_k/s_{k-1}}\leq q_{m-1}^{-1+s_k/s_{k-1}}.
\]
Thus, for any convergent $p_m/q_m$ of the continued fraction of $\xi$, we have
\begin{equation}\label{i1}
\left|\xi-\dfrac{p_m}{q_m}\right |>\frac{1}{3b_{m+1}q_m^2}\geq \frac{1}{q_{m}^{1+s_{k+1}/s_k}}.
\end{equation}
Aiming for a contradiction, suppose that $\xi\in \mathbb{L}^{(\nu+\epsilon)}$, then there exists infinitely many rational numbers $c_i/d_i$, with $d_i\geq 2$ and a sequence $(\theta_i)_i$, with $\theta_i\geq 3$, such that
\begin{equation}\label{i2}
\left|\xi-\dfrac{c_i}{d_i}\right |<\dfrac{1}{d_i^{\theta_i}},\ \forall i,
\end{equation}
where $\lim_{i\to \infty}\theta_i/d_i^{\nu+\e}=\infty$. Since $1/d_i^{\theta_i}\leq 1/(2d_i^2)$, then $c_i/d_i$ must be a convergent of the continued fraction of $\xi$, say $p_{m(i)}/q_{m(i)}$, where $r(k)<m(i)\leq r(k+1)$, for some $k$ (note that $d_i=q_{m(i)}$). Therefore, by combining (\ref{i1}) and (\ref{i2}), we arrive at 
\[
\frac{1}{3q_{m(i)}^{s_{k+1}/s_k}}<\left|\xi-\dfrac{p_{m(i)}}{q_{m(i)}}\right |<\dfrac{1}{q_{m(i)}^{\theta_i}}
\]
and so $\theta_i<2s_{k+1}/s_{k}$. By reordering indexes if necessary, we may suppose that $i\mapsto m(i)$ and $k\mapsto r(k)$ are increasing functions (thus $i$ tends to infinity if and only if $k$ does as well). By dividing both sides of the previous inequality by $d_i^{\nu+\epsilon}$, we obtain
\[
\frac{\theta_i}{d_i^{\nu+\epsilon}}<\frac{2\cdot 10^{\nu s_k}s_k}{10^{s_k(\nu+\epsilon)}}=\frac{2s_k}{10^{\epsilon s_k}},
\]
where we used that $d_i=q_{m(i)}>q_{r(k)}=10^{s_k}$. However, the inequality $\theta_i/d_i^{\nu+\epsilon}<2s_k/10^{\epsilon s_k}$ leads to a contradiction, since its left-hand side tends to infinity while its right-hand side tends to zero as $i$ and so $k$ tends to $\infty$. Therefore, $\xi$ is a $\nu$-Liouville number as desired. An easy modification of this proof shows that $\sum_{n\geq 1}a_n 10^{-s_n}\in \mathbb{L}^{(\nu)}$, for any choice of $a_n \in \{1, 2\}$. Hence, there exist uncountably many $\nu$-Liouville numbers, which completes the proof.\qed

\section{Proof of Theorem \ref{main1}}

\subsection{Auxiliary results}

We start by stating two ingredients for the proof (in all that follows, we use the notation $[a,b]=\{a,a+1,\ldots, b\}$, for integers $a<b$). 

The first one is due to \' I\c{c}en \cite{Ic}, which provides an inequality relating the {\it height} of algebraic numbers (the height of $\gamma\in \QQ$, denoted by $H(\gamma)$ is the maximum of the absolute values of the coefficients of the minimal polynomial of $\gamma$ (over $\Z$)).

\begin{lem}\label{l0}
Let $\alpha_1,\ldots,\alpha_n$ be algebraic numbers belonging to a $g$-degree number field $K$, and let $F(y,x_1,\ldots,x_n)$ be a polynomial with integer coefficients and with degree $d>1$ in the variable $y$. If $\eta$ is an algebraic number such that $F(\eta,\alpha_1,\ldots,\alpha_n)=0$, then the degree of $\eta$ is at most $dg$, and 
\begin{equation} \label{içendes}
H(\eta) \leq 3^{2dg + (l_1+\cdots + l_k)g} \cdot H^g \cdot H(\a_1)^{l_1g} \cdots H(\a_n)^{l_ng}, 
\end{equation}
where $H$ is the maximum of the absolute values of the coefficients of $F$ (the height of $F$), $l_i$ is the degree of $F$ in the variable $x_i$, for $i\in [1,n]$.
\end{lem}

Also, the lower bounds for linear form in logarithms {\it \` a la Baker} lie at the heart of our approach. Among the many extensions (and variations) of the original Baker's bounds, we choose to use the lower bound due to Mignotte and Waldschmidt (see Remark after Theorem in \cite{mig}), which is the most convenient for our purpose. For the sake of preciseness, we state it as a lemma:

\begin{lem}\label{lem:1}
Let $\beta,\a_1,\a_2$ be nonzero algebraic numbers and let $D=[\Q(\beta,\a_1,\a_2):\Q]$. Set $A_j\geq H(\a_j)$ ($j=1,2$) and $B\geq \max\{H(\b), e^{D}\}$. If $\Lambda = \beta\log \a_1-\log \a_2$ is not zero, then
\[
|\Lambda|>\exp\left(-5\cdot 10^{10}D^4(\log A_1)(\log A_2)T^2\right),
\]
where $T:=\log B+\log\log A_1+\log\log A_2+\log D$.
\end{lem}
Throughout this text, $\log z$ denotes the main value of the complex logarithm function. 

In the next section, each given inequality holds for all sufficiently large $k$ (unless stated otherwise), which makes valid all the previous ones.

\subsection{The proof}

First, we note that if $R(X)\equiv 0$, then $P(\xi)^{Q(\xi)^{R(\xi)}}=P(\xi)$ is transcendental, for all $\xi\in \mathbb{L}^{(\nu)}$ (since a non-constant polynomial takes the set of transcendental numbers into itself).

Aiming for a contradiction, suppose that $P(\xi)^{Q(\xi)^{R(\xi)}}$ is an algebraic number, say $\g$. In particular, $Q(\xi)^{R(\xi)}\log P(\xi) = \log \g$ (here we used that $\g\neq 0$, since $P(\xi)\neq 0$, because $\xi\not\in \QQ$). In what follows, $m, n$ and $r$ denote the degree of the polynomials $P(x), Q(x)$ and $R(x)$, respectively.

Since $\xi$ is a $\nu$-Liouville number, then there exist infinitely many rational numbers $\pq$ (with $q_k > 1$) and an increasing sequence $(\o_k)_k$ belonging to $\R_{>0}$ such that
\begin{equation} \label{ineq1}
    0<\abs{\xi - \frac{p_k}{q_k}}<\frac{1}{q_k^{\o_k}},
\end{equation}
for all $k\geq 1$, where $\lim_{k\to \infty}\o_k/q_k^{\nu}=\infty$. We may suppose (by reordering the indexes if necessary) that $\pq\in \Omega:=\{z\in \C: P(z)Q(z)R(z)\neq 0\ \mbox{and}\ P(z)\neq 1\}$, for all $k\geq 1$. Now, we claim 
\[
\Lambda_{k}:= Q(\pq)^{R(\pq)}\log P(\pq) - \log{\gamma}
\]
is nonzero for all sufficiently large $k$. In fact, on the contrary, i.e., if $\Lambda_{k} = 0$ for all $k$ belonging to an infinite set of positive integers, say $N^{\prime}$, then the function $h:\Omega\to \C$ given by $h(z):=Q(z)^{R(z)}\log P(z)-\log \g$ vanishes in the set $\{\pq : k\in N^{\prime}\}$ which has an accumulation point (namely $\xi$). By the Identity Theorem for Analytic Functions, one has that $h(z)=0$, for all $z\in \Omega$, which implies that the function $\tilde{h}(z):=h(z)+\log \g$ is the constant function. On the other hand, either $\Omega^{+}:=\{z\in \Omega\cap \Z_{>0}:R(z)\geq 1\}$ or $\Omega^{-}:=\{z\in \Omega\cap \Z_{>0}:R(z)\leq -1\}$ must be an infinite set (actually, only one of them has this property). If $\Omega^{+}$ is infinite, we use that
\[
|\tilde{h}(z)| \geq |Q(z)|^{R(z)}(\log |P(z)|+i\Arg(P(z)))\geq |Q(z)|(\log |P(z)|-2\pi),
\]
to infer that $|\tilde{h}(z)|$ tends to infinity as $z\to \infty$ ($z\in \Omega^{+}$). This contradicts the fact that $\tilde{h}(z)$ is constant. For the case in which $\Omega^{-}$ is an infinite set, we use that $|P(z)|< |z|^{m+1}$ and $|Q(z)|>(|c|/2)|z|^n$ for all sufficiently large $z\in \Omega^{-}$ (here $c$ is the leading coefficient of $Q(z)$). Thus
\[
|\tilde{h}(z)| \leq \frac{1}{|Q(z)|^{-R(z)}}|\log P(z)|\leq \dfrac{(m+1)\log z + 2\pi}{|c/2|z^n}
\]
yielding that $|\tilde{h}(z)|$ tends to $0$ as $z\to \infty$ ($z\in \Omega^{-}$). Therefore $\tilde{h}(z)$ must be the zero function which is absurd, since $\tilde{h}(z)\neq 0$, for all $z\in \Omega$.

In conclusion, $\Lambda_k$ is nonzero, for all sufficiently large $k$. Hence, all hypotheses of Lemma \ref{lem:1} are fulfilled for $\Lambda_k$ and therefore
\begin{equation} \label{ineq1.4}
    \abs{\Lambda_k}> \exp(-5\cdot 10^{10}D_k^4(\log A_{1,k})(\log A_{2,k})T_k^2),
\end{equation}
where $A_{1,k}\geq H(P(p_k/q_k)), A_{2,k}:=H(\g)$, $T_k\geq \log B_k+\log \log A_{1,k}+\log \log A_{2,k}+\log D_k$ (and $B_k=H(Q(\pq)^{R(\pq)})$). Now, we need to make some estimates in order to choose $A_{1,k}, T_k$, and $D_k$.

First, in Lemma \ref{l0}, set $F(Y,X):=Y-P(X)\in \Z[X,Y]$, then $F(P(\pq),\pq)=0$ and so, by that lemma, one deduces that $H(P(\pq))\leq 3^{m+2}(1+|\xi|)^mH(P)q_k^m$. Thus, we can take $A_k:=c_0q_k^m$, where $c_0=3^{m+2}(1+|\xi|)^mH(P)$. 

Next, by setting $\theta_k:=q_k^rR(\pq)$ (which is an integer), we note that $Q(\pq)^{R(\pq)}$ is a root of the polynomial $q_k^{n\theta_k}X^{q_k^r}-(q_k^nQ(\pq))^{\theta_k}\in \Q[X]$ and so it has degree at most $q_k^r$. Thus, we can choose $D_k:=q_k^r$ (since $P(\pq)\in \Q$). Furthermore, one of the polynomials $q_k^{n\theta_k}X^{q_k^r}-(q_k^nQ(\pq))^{\theta_k}$ and $(q_k^nQ(\pq))^{\theta_k}X^{q_k^r}-q_k^{n\theta_k}$ has integers coefficients (depending on the sign of $R(\pq)$). Thus, this polynomial is divisible by the minimal polynomial of $Q(\pq)^{R(\pq)}$ (over $\Z$). In particular, after some manipulations, we arrive at
\[
H(Q(\pq)^{R(\pq)})\leq q_k^{c_1q_k^r},
\]
where $c_1:=3n(1+|\xi|)^{mnr}\max\{H(P),H(Q),H(R)\}$. Thus, we can choose $B_k:=q_k^{c_1q_k^r}$ and $T_k:=c_2q_k^r\log q_k$ (where $c_2:=c_1+\log c_0+m+\log H(\g)+r$).

By using these estimates in (\ref{ineq1.4}), we deduce (after some straightforward calculations) that
\begin{equation} \label{ineq1.5}
    \abs{\Lambda_k}>\exp(-c_3q_k^{6r}(\log q_k)^3),
\end{equation}
where $c_3:=5\cdot 10^{10}mc_2^2(\log c_0)(\log H(\g))$.

On the other hand, we have
\begin{eqnarray} \label{ine1}
\abs{\Lambda_k} & = & \abs{Q(\pq)^{R(\pq)}\log P(\pq)-Q(\xi)^{R(\xi)}\log P(\xi)} \nonumber\\
 & \leq & \abs{\log P(\pq)}\abs{Q(\xi)^{R(\xi)}-Q(\pq)^{R(\pq)}}\\
 & & +\abs{Q(\xi)^{R(\xi)}}\abs{\log P(\xi)-\log P(\pq)}\nonumber
\end{eqnarray}
where we used that $Q(\xi)^{R(\xi)}\log P(\xi) = \log \g$ together with the triangle inequality. It is well-known that for a holomorphic function $\phi: U\to \C$, if $z_1,z_2\in U$ are such that the the line segment with endpoints $z_1, z_2$, denoted by, $[z_1,z_2]$, belongs to $U$, then $|\phi(z_1)-\phi(z_2)|\leq |z_1-z_2|\cdot \max\{|\phi'(\zeta)|: \zeta\in [z_1,z_2]\}$. By using this fact for the holomorphic functions $f:\Omega\to \C$ and $g:\Omega \to \C$ given by
\begin{center}
    $f(z):=\log P(z)$ and $g(z):=Q(z)^{R(z)}$,
\end{center}
we obtain 
\begin{center}
    $\abs{\log P(\xi)-\log P(\pq)}\leq c_4\abs{\xi - \frac{p_k}{q_k}}$ and $\abs{\log P(\xi)-\log P(\pq)}\leq c_5\abs{\xi - \frac{p_k}{q_k}}$,
\end{center}
where $c_4:=\sup_{\zeta\in I_k}|f'(\zeta)|$ and $c_5:=\sup_{\zeta\in I_k}|g'(\zeta)|$ (here $I_k$ is the closed interval with endpoints $\xi$ and $\pq$). So, we get from (\ref{ine1}) that
\begin{equation}\label{ineq1}
    \abs{\Lambda_k} < c_6\abs{\xi - \frac{p_k}{q_k}} < \dfrac{c_6}{q_k^{\o_k}},
\end{equation}
for all sufficiently large $k$ (say $k\geq \kappa_1$), where $c_6:=|\log P(\xi)|+|Q(\xi)^{R(\xi)}|+c_4+c_5+1$.

Now, by combining \eqref{ineq1.5} and \eqref{ineq1}, we obtain
\[
c_6q_k^{-\o_k}>\exp(-c_3q_k^{6r}(\log q_k)^3)
\]
yielding 
\[
\o_k<2c_3q_k^{6r}(\log q_k)^2
\]
for all sufficiently large $k$. However, there exist $\epsilon>0$ such that $\nu=6r+\epsilon$ and so the last inequality can be rewritten as
\[
\dfrac{\o_k}{q_k^{6r+\epsilon}}<2c_3\dfrac{(\log q_k)^2}{q_k^{\epsilon}}
\]
This leads to an absurd, since the left-hand side of the previous inequality tends to infinity, while its right-hand side tends to zero when $k\to \infty$. In conclusion, the number $P(\xi)^{Q(\xi)^{R(\xi)}}$ is transcendental and the proof is complete. 
\qed

\section{Further comments, results, and a conjecture}

We close this article by addressing a discussion on some more technical results and possible problems for further research. In what follows, the implied constants in $\ll$ and $\gg$ do not depend on $k$. 

First, we point out that when $P(X)$ is a constant, it can be replaced by some classes of transcendental numbers. Indeed, let us define the set $\mathcal{T}$ by
\[
\mathcal{T}:=\{e^{\beta_0}\a_1^{\b_1}\cdots \a_n^{\b_n} : n\geq 1\ \mbox{and}\ (\beta_0,\ldots, \b_n,\a_1,\ldots,\a_n)\in \QQ^{2n+1}\}
\]

Note that the preceding set contains all algebraic numbers (case $(n,\b_0,\b_1)=(1,0,1)$) as well as the families of transcendental numbers proved by Hermite and Lindemann (case $(n,\a_1)=(1,1)$), Gelfond and Schneider (case $(n,\b_0)=(1,0)$) and Baker ($n>1$, under some mild hypotheses).

We have the following result:

\begin{prop}\label{p1}
Let $(P(X), Q(X), R(X))$ be a $\QQ$-admissible triple of polynomials with $R(X)\in \Q[X]$. If $\zeta\in \mathcal{T}$, then the numbers
\begin{center}
    $\zeta^{Q(\xi)^{R(\xi)}}$ and $P(\xi)^{Q(\xi)^{R(\xi)}}$
\end{center}
do not belong to $\mathcal{T}$, for all $\xi\in \mathbb{L}^{(\infty)}$. In particular, $\xi^{\alpha}, \xi^{\xi^r}, e^{\xi^{\xi}}$ and $\beta^{\alpha\xi^{\xi}}$ are transcendental numbers, whenever $\alpha\in \QQ\backslash \{0\}$, $\beta\in \QQ\backslash \{0,1\}, r\in \Q$ and $\xi\in \uL$.
\end{prop} 

Before the proof, we state the following lower bound for linear forms due to Baker (see \cite{baker}). For the sake of preciseness, we state it as a lemma:

\begin{lem}\label{lem:B}
Let $n> 1$ be an integer and let $\a_1,\ldots, 
\a_{n}, \b_0,\b_1,\ldots, \b_{n}$ be nonzero algebraic numbers. Set $A=\prod_{i=1}^n\log (\max\{H(\a_i), 4\})$ and $B=\max_{i\in [1,n]}\{H(\b_i),4\}$. Let $D=[\mathbb{Q}(\a_1,\ldots, 
\a_n, \b_0,\ldots, \b_{n}):\mathbb{Q}]$. Then, 
\[
|\b_0+\b_1\log \a_1+\cdots +\b_{n}\log \a_{n}|>(AB)^{-(16nD)^{200n}A\log A},
\]
provided that the previous linear form does not vanish.
\end{lem}

\begin{proof}[Proof of Proposition \ref{p1}]
We shall proceed along the same lines as in the proof of Theorem \ref{main1}. In fact, striving for a contradiction, suppose the existence of $\zeta_1, \zeta_2\in \mathcal{T}$ and $\xi\in \mathbb{L}^{(\infty)}$ such that $\zeta_1^{Q(\xi)^{R(\xi)}}=\zeta_2$. By proceeding along the same lines as before, we obtain that $|\Lambda_k|\ll q_k^{-\omega_k}$, where $\Lambda_k:=Q(\pq)^{R(\pq)}\log \zeta_1 - \log \zeta_2$ is a linear form in logarithms of algebraic numbers. By Lemma \ref{lem:B}, after some calculations, one has $|\Lambda_k|\gg \exp(-q_k^{O(1)})$. We then combine the upper and lower bounds for $|\Lambda|$ to obtain $\o_k\ll q_k^{O(1)}/\log q_k$ which contradicts the fact that $\lim_{k\to \infty}\o_k/q_k^{\nu}=\infty$, for all $\nu>0$ (since $\xi$ is an $\infty$-Liouville number). 

For the transcendence of $P(\xi)^{Q(\xi)^{R(\xi)}}$, we simply mimic the same approach. The novelty here lies
in the fact that $P(X)$ and $Q(X)$ have algebraic coefficients. So, it suffices to prove that the degrees and heights of the numbers $P(\pq)$ and $\gamma_k:=Q(\pq)^{R(\pq)}$ have the reasonable growth order compared to the case of integer coefficients (by reasonable, we mean that if the integer case is $\leq f(k)^{g(k)}$, then the algebraic one is $O(f(k)^{O(g(k))})$). To prove that, let $\mathbb{K}$ be the $d$-degree number field which contains all the coefficients of $P(X)$ and $Q(X)$. Clearly, the degree of $P(\pq)$ is at most $d$. By supposing, without loss of generality, that the integer $q_k^rR(\pq)$ is positive, then $\gamma_k^{q_k^r}$ has degree at most $d$ and so $\gamma_k$ has degree $\leq dq_k^r$. For the heights, we shall use the following relations (which can be found, for instance, in \cite[p. 325]{DGaa}): Let $y_1,\ldots, y_{\ell}$ be algebraic numbers of degree $m_1,\ldots, m_{\ell}$, respectively. Then 
\begin{itemize}
\item $(2^{m_1n}\sqrt{2m_1})^{-1}H(y_1)^{n}\leq H(y_1^n)\leq e^{O(n)} H(y_1)^{n}$;
\item $H(y_1y_2)\ll (H(y_1)H(y_2))^{m_1m_2}$;
\item $H(y_1+\cdots +y_{\ell})\leq e^{O(\ell)}(H(y_1)\cdots H(y_{\ell}))^{m_1\cdots m_{\ell}}$,
\end{itemize}
where the implied constants depend only on the degree of the algebraic numbers. Therefore, by using these estimates, it is straightforward to show that $H(P(\pq))\ll q_k^{m^2d^{m+2}}$ and $H(\gamma_k)\ll q_k^{O(q_k^r)}$. Hence, we confirm that these estimates are reasonable (in our sense) which ensures that the condition of $\xi$ being an $\infty$-Liouville number is enough to complete the proof (we leave the details to the reader). 
\end{proof}

We now continue by recalling the following definitions:

\begin{deff} \label{defh}
Let $k\geq 1$ be an integer. The power tower function of order $k$, $h_k(x)$, is defined for all real number $x>0$ as
\begin{center}
$h_k(x)=x^{\,x^{\,\cdot^{\cdot^{\cdot^{x}}}}}$ ($k$ times)
\end{center}

For all $x\in [e^{-e}, e^{1/e}]$, $h_k(x)$ converges (as $k\to \infty$) to a function $h(x)$ which is named as the infinite power tower function (see Eisenstein \cite{eis}). In other words,
\begin{equation*}
 h(x) = x^{\,x^{\,x^{\,\cdot^{\cdot^{\cdot}}}}}.
\end{equation*}
\end{deff}

The previous definitions imply in the following functional equations
\begin{equation}\label{func}
h(x) = x^{\,h(x)},\ h_{k+1}(x)=x^{h_k(x)}\ \mbox{and}\ h_{k+2}(x)=x^{x^{h_k(x)}}.
\end{equation}

\begin{rem}\label{remLast}
    We remark that $h_1(\xi), h_2(\xi)$ and $h_3(\xi)$ are transcendental numbers whenever $\xi\in \uL$ (by \cite[Prop. 3.2]{MO} and Corollary \ref{main3}).
\end{rem}

Turning back to our problems, we point out that there are at least two directions in order to generalize the previous results, namely, by increasing the power tower of Liouville numbers or by replacing Liouville numbers by $U_m$-numbers (we refer to \cite{bugeaud} for their definition). However, our argument is invalid to deal even with the transcendence of $h_4(\xi)$ and $h_3(\zeta)$, where $\xi$ is a Liouville number and $\zeta$ is a $U_m$-number (with $m>1$). Indeed, in both cases, at least one of the coefficients of the linear forms will be $h_3(r)$ and $h_2(\a)$, for non-integer rational numbers $r$ and non-rational algebraic numbers $\a$. Since $h_2(r)\not\in \Q$, for any $r\in \Q\backslash \Z$, the Gelfond-Schneider theorem implies in the transcendence of both $h_3(r)$ and $h_2(\a)$. Therefore, the linear form $\Lambda$ will not have algebraic coefficients, so it is useless in the sense of Baker's theory.

However, there are some consequences of Proposition \ref{p1} concerning the previous functions:

\begin{cor}\label{clast}
Let $\xi$ be a positive $\infty$-Liouville number. We have:
\begin{itemize}
    \item[(i)] If $\xi\in [e^{-e}, e^{1/e}]$, then $h(\xi)$ is a transcendental number;
    \item[(ii)] For any $k\geq 1$, at least one of the numbers $h_{k}(\xi)$ or $h_{k+1}(\xi)$ is transcendental.
    \item[(iii)] If $h_k(\xi)$ is a rational number, then both $h_{k+1}(\xi)$ and $h_{k+2}(\xi)$ are transcendental numbers.
\end{itemize}
\end{cor}
\begin{proof}
    The proofs are direct combinations of Proposition \ref{p1} and (\ref{func}). In fact, if $h(\xi)=\gamma\in \QQ\backslash \{0\}$, then $\xi^{\gamma}$ is transcendental contradicting the fact that $\xi^{\gamma}=\gamma$ (this proves (i)). For (ii), we use that if $h_k(\xi)=\gamma$ is algebraic (and so nonzero), then $h_{k+1}(\xi)=\xi^{\gamma}$ is transcendental. To prove (iii), if $h_k(\xi)=r$ is a rational number, then $h_{k+1}(\xi)$ is transcendental (by item (ii)) as well as $h_{k+2}(\xi)=\xi^{\xi^{r}}$.
\end{proof}

We finish by posing a conjecture, which we do believe to be true. 
\begin{conj}
    If $\xi$ is an ultra-Liouville number, then $h_k(\xi)$ is transcendental, for all $k\geq 1$.
\end{conj}

A partial evidence to support this conjecture is that Corollary \ref{clast} (ii) implies in the existence of infinitely many transcendental numbers among $h_4(\xi), h_5(\xi),\ldots$ whenever $\xi\in \mathbb{L}^{(\infty)}$.

\section*{Acknowledgement}
D.M. was supported, in part, by CNPq-Brazil. M.O.  was supported by the Research Support Grant, Departamento de Matemática, Universidade Estadual do Maranhão. P.T. was supported by the Project of Excellence, Faculty of Science, University of Hradec Kr\'alov\'e, No. 2210/2023-2024. Part of this work was done during a very enjoyable visit of D.M. to IMPA (Rio de Janeiro) in July 2023.



\begin{thebibliography}{9999}

\bibitem{Baker} A. Baker. Transcendental number theory. Cambridge university press, 1975.

\bibitem{baker} A. Baker, The theory of linear forms in logarithms, ``Advances in transcendence theory". -- London and New York, Academic Press, 1977.

\bibitem{bugeaud} Y. Bugeaud, {\it Approximation by Algebraic Numbers},  Cambridge Tracts in Mathematics, {\bf 160}. Cambridge University Press, Cambridge (2004).

\bibitem{bull} T. Chalebgwa, S. Morris, $\sin, \cos, \exp$ and $\log$ of Liouville numbers, to appear in {\it Bull. Aust. Math. Soc.}

\bibitem{cij} P. Cijsouw, {\it Transcendence Measures}. Doctoral dissertation, Universiteit van Amsterdam, 1972, 107 pp.

\bibitem{Ic} O. \c{S}. \.{I}\c{c}en, \"{U}ber die funktionswerte der {$p$}-adisch elliptischen funktionen. {I}, {II}, {\it \.{I}stanbul \"{U}niv. Fen Fak. Mecm. Ser. A}, {\bf 36} (1971), 53--87 (1974); ibid. {\bf 35} (1970), 139--166 (1972).

\bibitem{eis} G. Eisenstein, Entwicklung von $\alpha^{\alpha ^{\alpha^{\cdot^{\cdot^{\cdot}}}}}$, \textit{J. Reine Angew. Math.} \textbf{28} (1844), 49--52.

\bibitem{koksma} J. F. Koksma, {\"U}ber die {M}ahlersche Klasseneinteilung der Transzendenten Zahlen und die Approximation Komplexer Zahlen durch Algebraische Zahlen, {\it Monatsh. Math.}, {\bf 48} (1939), 176--189.

\bibitem{Tang} K. S. Kumar, R. Thangadurai, M. Waldschmidt, Liouville
numbers and Schanuel's Conjecture, {\it Archiv der Mathematik}, {\bf 102}
(2014), 59--70.

\bibitem{liouville} J. Liouville, Sur des classes tr{\` e}s-{\'e}tendues de quantit{\'e}s dont la Valeur n'est ni alg{\'e}brique ni m{\^e}me r{\'e}ductible {\`a} des irrationnelles alg{\'e}briques, {\it C. R. Acad. Sci. Paris}, {\bf 18} (1844), 883--885.

\bibitem{mahler} K. Mahler, Zur approximation der exponentialfunktion und des logarithmus. Teil {I}., {\it J. Reine Angew. Math.}, {\bf 166} (1932), 118--150.

\bibitem{mah2} K. Mahler, On the approximation of logarithms of algebraic numbers, {\it Philos. Trans. Roy. Soc. London Ser. A} {\bf 245} (1952), 371--398.

\bibitem{mah3} K. Mahler, Some suggestions for further research, {\it Bull. Austral. Math. Soc.,} {\bf 29} (1984),
101--108.

\bibitem{mai} E. Maillet, \textit{Introduction \` a la Th\' eorie des Nombres Transcendants et des Propri\' et\' es
Arithm\' etiques des Fonctions}. Gauthier-Villars, Paris (1906).

\bibitem{digu} D. Marques, C. G. Moreira, On a question proposed by K. Mahler concerning Liouville numbers. {\it Bull. Aust. Math. Soc.}, {\bf 91} (2015), 29--33.

\bibitem{MO} D. Marques, M. Oliveira, On the transcendence of some powers related to $U$-numbers. {\it Bull. Braz. Math. Soc.}, {\bf 54} Article number: 25 (2023).

\bibitem{DGaa} D. Marques, C.G. Moreira, On exceptional sets of transcendental functions with integer coefficients: solution of a Mahler’s problem, {\it Acta Arith.} {\bf 192} (2020), 313--327.

\bibitem{mig} Mignotte, M. Waldschmidt, M. Linear forms in two logarithms and Schneider's method. {\it Math. Ann.,} {\bf 231} (1978), 241--267.

\bibitem{schi} W. Schmidt, $T$-numbers do exist, \textit{Symposia Math.} \textbf{6} (1970), 3--26.

\end{thebibliography}
\end{document}